\documentclass[11pt,x11names]{article}
\usepackage{latexsym}
\usepackage[dvips]{graphicx}
\usepackage{amsfonts,amssymb,amsmath}
\usepackage{pdfsync}
\usepackage{epsfig}
\usepackage{geometry}
\usepackage{color}
\usepackage{pstricks}
\usepackage{setspace}
\usepackage{comment}
\usepackage{enumerate}
\usepackage{cite}
\usepackage{hyperref}
\usepackage{graphicx}
\usepackage{epstopdf}
\usepackage{float}
\usepackage{cases}
\usepackage{blindtext}

\newcommand{\N}{\mathbb{N}}
\newcommand{\NN}{\mathbb{N}}

\newcommand{\R}{\mathbb{R}}

\newcommand{\RR}{\mathbb{R}}

\newcommand{\cali}{\mathcal}

\newtheorem{theorem}{Theorem}
\newtheorem{proposition}[theorem]{Proposition}
\newtheorem{lemma}[theorem]{Lemma}
\newtheorem{corollary}[theorem]{Corollary}

\newtheorem{definition}{Definition}

\newtheorem{remark}[theorem]{Remark}

\newcommand{\f}{\varphi}

\newenvironment{proof}[1][]{\noindent {\bf Proof #1:\;}}{\hfill $\Box$\\}

\DeclareMathOperator{\crit}{crit}

\DeclareMathOperator{\argmin}{argmin}

\DeclareMathOperator{\dist}{dist}

\DeclareMathOperator{\dom}{dom}

\DeclareMathOperator{\prox}{prox}

\DeclareMathOperator{\Graph}{graph}

\begin{document}

\title{Extragradient Method in Optimization: Convergence and Complexity}


\author {Trong Phong Nguyen\footnote{TSE (GREMAQ, Universit\'{e} Toulouse I Capitole), Manufacture des Tabacs, 21 all\'{e}e de Brienne, 31015 Toulouse, Cedex 06, France. Email: trong-phong.nguyen@ut-capitole.fr and Centro de Modelamiento Matemático (UMI 2807, CNRS), Universidad de Chile, Beauchef 851, Casilla 170-3, Santiago 3, Chile} \and Edouard Pauwels \footnote{IRIT-UPS, 118 route de Narbonne, 31062 Toulouse, France. Email: edouard.pauwels@irit.fr (Edouard Pauwels)} \and  \'Emile Richard \footnote{Amazon,	ricemile@amazon.com} \and Bruce W. Suter\footnote{Air Force. Research Laboratory / RITB, Rome, NY, United States of America. E-mail: bruce.suter@us.af.mil}}


\maketitle

\begin{abstract}
We consider the extragradient method to minimize the sum of two functions, the first one being smooth  and the second being  convex. Under the Kurdyka-\L ojasiewicz assumption, we prove that the sequence produced by the extragradient method converges to a critical point of the problem and has finite length. The analysis is extended to the case when both functions are convex. We provide, in this case,  a sublinear convergence rate, as for gradient-based methods. Furthermore, we show that the recent \textit{small-prox} complexity result can be applied to this method. Considering the extragradient method is an occasion to describe an exact line search scheme for proximal decomposition methods. We provide details for the implementation of this scheme for the one norm regularized least squares problem and demonstrate numerical results which suggest that combining nonaccelerated methods with exact line search can be a competitive choice.
\end{abstract}
\section{Introduction}
We introduce a new optimization method for approximating a global minimum of a composite objective function, i.e., a function formed as the sum of a smooth function and a simple nonsmooth convex function.

This class of problems is rich enough to encompass many smooth/nonsmooth, convex/nonconvex optimization problems considered in practice. Applications can be found in various fields throughout science and engineering, including signal/image processing \cite{CP} and machine learning \cite{tibshirani1996regression}. Successful algorithms for these types of problems include for example fast iterative shrinkage-thresholding algorithm (FISTA) method \cite{BT08} and forward-backward splitting method \cite{Waj}. The goal of this paper is to investigate to which extent extragradient method can be used to tackle similar problems.

The extragradient method was initially proposed by Korpelevich \cite{korpe} and  it has become a classical method for solving variational inequality problems. For optimization problems, this method generates a sequence of estimates based on two projected gradient steps at each iteration.

After Korpelevich's work, a number of authors extended his extragradient method for variational inequality problems (for example, see \cite{censor,Svai}). In the context of convex constrained optimization, \cite{LuoTseng} considered the performances of the extragradient method under error bounds assumptions. In this setting, Luo and Tseng have described asymptotic linear convergence of the extragradient method applied to constrained problems. To our knowledge, this is the only attempt to analyse the method in an optimization setting.

A  distinguishing feature of the extragradient method is its use of an additional projected gradient step, which can be seen as a guide during the optimization process. Intuitively, this additional iteration allows us to  \emph{examine} the geometry of the problem and take into account its curvature information, one of the most important bottlenecks for first order methods. Motivated by this observation, our goal is to extend and understand further the extragradient method in the specific setting of optimization. Apart from the work of Luo and Tseng, the literature on this topic is quite scarce. Moreover, the nonconvex case is not considered at all.

We  combined the work of \cite{korpe,LuoTseng} and some recent extensions for first-order descent methods, (see \cite{attbol,AttBolSva,BST,eb&kl}), to propose the extented extragradient method (EEG for short) to tackle the problem of minimizing a composite objective function. The classical extragradient method relies on orthogonal projections. We extend it by considering more general nonsmooth functions, and using proximal gradient steps at each iteration. An important challenge in this context is to balance the magnitude of  the two associated parameters to maintain desirable convergence properties. We devise conditions which allow to prove convergence of this method in the nonconvex case. In addition, we describe two different rates of convergence in the convex setting.

Following \cite{attbol,AttBolSva,BST,eb&kl} we heavily rely on the Kurdyka-\L ojasiewicz (KL for short) inequality to study the nonconvex setting. The KL inequality \cite{Loja63,Kur98} has a long history in convergence analysis and smooth optimization. Recent generalizations in the seminal works \cite{BolDanLew1,BolDanLewShi07} have shown the important versatility of this approach as the inequality holds true for the vast majority of models encountered in practice, including nonsmooth and extended valued functions. This opened the possibility to devise general and abstract convergence results for first order methods \cite{AttBolSva,BST}, which constitute an important ingredient of our analysis. Based on this approach, we derive a general convergence result for the proposed EEG method.

In the convex case, we  focus on global convergence rates. We first describe a sublinear convergence rate in terms of objective function. This is related to classical results from the analysis of first order methods in convex optimization, see for example the analysis of forward-backward splitting method in \cite{BT08}. Furthermore, we show that the \textit{small-prox} result of \cite{eb&kl} also applies to EEG method which echoes the error bound framework of Luo and Tseng \cite{LuoTseng} and opens the door to more refined complexity results when further properties of the objective function are available.

As already mentioned, a distinguished aspect of the extragradient method is its use of an additional proximal gradient step at each iteration. The intuition behind this mechanism is the incorporation of curvature information  into the optimization process. It is expected that one of the effects of this additional step is to allow larger step sizes. With this in mind, we describe an exact line search variant of the method. Although computing exact line search is a nonconvex problem, potentially hard in the general case, we describe an active set method to tackle it for the specific and very popular case of the one norm regularized least squares problem (also known as the least absolute shrinkage and selection operator or LASSO). In this setting the computational overhead of exact line search is approximately equal to that of a gradient computation (discarding additional logarithmic terms).

On the practical side, we compare the performance of the proposed EEG method (and its line search variant) to those of FISTA and forward-backward splitting methods on the LASSO problem. The numerical results suggest that EEG combined with exact line search, constitutes a promising alternative which does not suffer too much from ill conditioning.
\paragraph{Structure of the paper.}
Section \ref{sec2}  introduces the problem and our main assumptions. We also recall important definitions and  notations which will be used throughout the text. Section \ref{sec3} contains the main convergence results of this paper. More precisely, in subsection \ref{subsec33}, we present the convergence and finite length property under the KL assumption in the nonconvex case. Subsection \ref{subsec34}, contains both a proof of sublinear convergence rate and the application of the \textit{small-prox} result for EEG method leading to improved complexity analysis under the KL assumption. Section \ref{sec4} describes exact line search for proximal gradient steps in the context of one norm regularized least squares and results from numerical experiments.

\section{Optimization Setting and Some Preliminaries}\label{sec2}
\subsection{Optimization Setting}
We are interested in solving minimization problems of the form
\begin{equation*}\label{P}\tag*{({\bf  P})}
\min_{x\in \R^n} \{F(x):=f(x)+g(x)\},
\end{equation*}
where $f,\, g$ are extended value functions from $\R^n$ to $\left]-\infty, +\infty\right]$. We make the following standing assumptions:
\begin{itemize}
	\item $\argmin F \ne \emptyset,$ and we note $F^*:=\min_{x\in\R^n} F(x)$.
	\item $g$ is a lower semi-continuous, convex, proper function.
	\item  $f$ is differentiable with $L$-Lipschitz continuous gradient, where $L>0$.
\end{itemize}
\subsection{ Nonsmooth Analysis}\label{nonsmooth}
In this subsection, we recall the definitions, notations and some well-known results from nonsmooth analysis which are going to be used throughout the paper. We will use notations from \cite{Rockafellar} (see also \cite{BauCom}). Let $h\colon\R^n\to\left]-\infty, +\infty\right]$ be a proper, lower-semicontinuous function. For each $x\in \dom h$, the Fr\'echet subdifferential of $h$ at $x$, written $\hat{\partial} h(x)$, is the set of vectors $u\in \R^n$ which satisfy
$$\liminf_{y\to x}\frac{h(y)-h(x)-\langle u, y-x\rangle}{\|x-y\|}\geq 0.$$
When $x\notin \dom h$, we set $\hat{\partial} h(x):=\emptyset$. We will use the following set
$$\Graph (\hat{\partial} h):= \left\{(x, u)\in \R^n \times \R^n\colon\:u \in \hat{\partial} h(x)  \right\}.$$
The (limiting) subdifferential of $h$ at $x\in \dom h$ is defined by the following closure process
$$\partial h(x):=\left\{u\in \R^n\colon \exists \left(x_m, u_m\right)_{m \in \N} \in \Graph (\hat{\partial} h)^\N, \, x_m\underset{m \to \infty}{\rightarrow} x,\,h(x_m)\underset{m \to \infty}{\rightarrow} h(x),\,u_m\underset{m \to \infty}{\rightarrow} u  \right\}.$$
$\Graph (\partial h)$ is defined similarly as $\Graph (\hat{\partial} h)$.
When $h$ is convex, the above definition coincides with the usual notion of subdifferential in convex analysis
$$\partial h(x):=\{u\in \R^n\colon\, h(y)\ge h(x)+\langle u,y-x\rangle\hbox{ for all }y\in \R^n\}.$$
Independently, from the definition, when  $h$ is smooth at $x$ then the subdifferential is a singleton, $\partial h(x)=\left\lbrace \nabla h(x)\right\rbrace$.\\

We can deduce from its definition the following closedness property of the subdifferential: if a sequence $(x_m,u_m)_{m\in \N} \in \Graph (\partial h)^\N$, converges to $(x,u)$, and $h(x_m)$ converges to $h(x)$ then $u\in \partial h(x)$. The set $\crit h:=\{x\in \R^n\colon 0\in \partial h(x)\}$ is called the set of critical points of $h$. In this nonsmooth context, Fermat's rule remains unchanged: A necessary condition for $x$ to be local minimizer of $h$ is that $x\in \crit h$ \cite[Theorem 10.1]{Rockafellar}.

Under our standing assumption, $f$ is a smooth function and we have subdifferential sum rule \cite[Exercise 10.10]{Rockafellar}
\begin{equation}\label{suru}
\partial (f+h)(x)=\nabla f(x)+\partial h(x).
\end{equation}
We recall a well known important property of smooth functions which have $L$-Lipschitz continuous gradient, (see \cite[Lemma 1.2.3]{nes}).
\begin{lemma}[Descent Lemma]\label{lem1}
	For any $x,y \in \R^n$, we have
	$$f(y)\leq f(x) +\langle y-x, \nabla f(x)\rangle +\frac{L}{2}\|x-y\|^2.$$
\end{lemma}
\noindent

For the rest of this paragraph, we suppose that $h$ is a convex function. Given $x\in \R^n$ and $t>0$, the proximal operator associated to $h$, which we denote by $\prox_{th}(x)$, is defined as the unique minimizer of function $\displaystyle y\longmapsto\ h(y)+\frac{1}{2t}\|y-x\|^2$, i.e:
$$\prox_{th}(x):=\argmin_{y\in \R^n} \left\lbrace h(y)+\frac{1}{2t}\|y-x\|^2\right\rbrace.$$
Using Fermat's Rule, $\prox_{th}(x)$ is characterized as the unique solution of the inclusion
$$\frac{x-\prox_{th}(x)}{t}\in \partial h\left(\prox_{th}(x)\right).$$
We recall that if $h$ is convex, then $\prox_{th}$ is nonexpansive, that is Lipschitz continuous with constant $1$ (see \cite[Proposition 12.27]{BauCom}). As an illustration, let $C\subset \R^n$ be a closed, convex and nonempty set, then $\prox_{i_C}$ is the orthogonal projection operator onto $C$.
The following property of the $\prox$ mapping will be used in the analysis, (see \cite[Lemma 1.4]{BT08}).
\begin{lemma}{\label{proximal}}
	Let $u\in \R^n,\,t>0$, and $v=\prox_{th}\left( u\right)$, then
	$$h(w)-h(v)\geq \frac{1}{2t}\left(\|u-v\|^2+\|w-v\|^2-\|u-w\|^2\right), \forall w\in \R^n .$$
\end{lemma}
\subsection{Nonsmooth Kurdyka-\L ojasiewicz Inequality}
In this subsection, we present the nonsmooth Kurdyka-\L ojasiewicz inequality introduced in \cite{BolDanLew1} (see also \cite{BolDanLewShi07,BolDanLeyMaz}, and the fundamental works \cite{Loja63,Kur98}). We note $\left[h<\mu\right]:=\{x\in \R^n\colon\:h(x)<\mu\}$ and $\left[\eta<h<\mu\right]:=\{x\in \R^n\colon\:\eta<h(x)<\mu\}$. Let $r_0>0$ and set
$$\cali K(r_0):=\left\lbrace\f \in C^0\left(\left[0,r_0\right[\right)\cap C^1\left(\left]0,r_0\right[\right),\, \f(0)=0,\, \f\text{ is concave and }\f'>0 \right\rbrace.$$
\begin{definition}
The function $h$ satisfies the {\em Kurdyka-\L ojasiewicz {\rm (KL)} inequality} (or has the KL {\em property}) locally at $\bar x\in \dom h$ if there exist $r_0>0$, $\f\in \cali K(r_0)$ and a neighborhood $U(\bar x)$ such that
\begin{equation}\label{defKL}
\f'\left(h(x)-h(\bar x)\right)\,\dist\left(0,\partial h(x)\right)\geq 1
\end{equation}
for all $x\in U(\bar x) \cap \left[h(\bar x)<h(x)<h(\bar x)+r_0\right]$. We say that $\f$ is a {\em desingularizing function} for $h$ at $\bar x$. The function $h$ has the KL property on $S$ if it does so at each point of $S$.
\end{definition}
When $h$ is smooth and $h(\bar x)=0$ then (\ref{defKL}) can be rewritten as
$$ \|\nabla(\f \circ h)\| \geq 1, \, \forall x\in U(\bar x) \cap \left[0<h(x)<r_0\right].$$
This inequality may be interpreted as follows: The function $h$ can be made sharp locally by a reparameterization of its values through a function $\f \in \cali K(r_0)$ for some $r_0 > 0$.

The KL inequality is obviously satisfied at any noncritical point $\bar x\in \dom h$ and will thus be useful only for critical points, $\bar x\in \crit h$. The {\em \L ojasiewicz gradient inequality} corresponds to the case when $\f(s)=c s^{1-\theta}$ for some $c>0$ and $\theta\in \left[0,1\right[$. The class of functions which satisfy KL inequality is extremely vast. Typical KL functions are semi-algebraic functions, but there exists many extensions, (see \cite{BolDanLew1}).

If $h$ has the KL property and admits the same desingularizing function $\varphi$ at {\em every point}, then we say that $\varphi$ is a {\em global} desingularizing function for $f$. The following lemma is given in \cite[Lemma 6]{BST}.
\begin{lemma}[Uniformized KL property]\label{KLunif}
	Let $\Omega$ be a compact set and let $h: \R^n\rightarrow \left]-\infty, \infty\right]$ be a proper and lower semicontinuous function. We assume that $h$ is constant on $\Omega$ and satisfies the KL property at each point of $\Omega$. Then, there exist $\varepsilon>0,\, \eta>0$ and $\varphi$ such that for all $\bar{x}\in\Omega$, one has
	$$\varphi'(h(x)-h(\bar{x}))\dist (0,\partial h(x))\geq 1,$$
	for all $x\in \left\lbrace x\in \R^n\colon \dist(x,\Omega)<\varepsilon\right\rbrace \cap \left[h(\bar{x})<h(x)<h(\bar{x})+\eta\right].$	
\end{lemma}

\section{Extragradient Method, Convergence and Complexity} \label{sec3}
\subsection{Extragradient Method}
We now describe our extragradient method dedicated to the minimization of problem \ref{P}. The method is defined, given an initial estimate $x_0\in \R^n$, by the following recursion, for $k\geq 1$,
\begin{numcases}{(EEG)}
y_{k}:=\prox_{s_k g}\left(x_k-s_k\nabla f(x_k)\right),\label{def1}\\
x_{k+1}:=\prox_{\alpha_k g}\left(x_k-\alpha_k\nabla f(y_k)\right)\label{def2}.
\end{numcases}
where $(s_k,\alpha_k)_{k\in \N}$ are positive step size sequences. We introduce relevant quantities, $s_-=\inf_{k\in \N} s_k, \,s^+=\sup_{k\in \N} s_k$, $\alpha_-=\inf_{k\in \N} \alpha_k$ and $\alpha^+=\sup_{k\in \N} \alpha_k$. Throughout the paper, we will consider the following condition on the two step size sequence,
$$({\bf C}):0<\alpha_-,\, 0<s_-,\,s^+<\frac{1}{L} \text{ and } 0< s_k\leq \alpha_k,\, \forall k\in \N.$$
Depending on the context, additional restrictions will be imposed on the step size sequences.
\subsection{Basic Properties}
We introduce in this subsection two technical properties of sequences produced by EEG method. These two technical properties, as abstract conditions introduced in tame nonconvex settings \cite{attbol,AttBolSva,BST},  allow us to prove the convergence of the sequences. We begin with a technical lemma.

\begin{lemma}
	Let $x\in \R^n$, $y \in \RR^n$, $t>0$, and $p=\prox_{tg}\left( x - t \nabla f(y)\right)$, then, for any $z \in \RR^n$, we have
	\begin{description}
\item[(i)]$\displaystyle	F(z) - F(p) \geq \left( \frac{1}{2t} - \frac{L}{2} \right) \|p - z \|^2 + \frac{1}{2t}\left( \|x - p\|^2 - \|x - z\|^2 \right) + \left\langle p-z, \nabla f(y) - \nabla f(z) \right\rangle.$
\item[(ii)] $\displaystyle F(z) - F(p) \geq \frac{1}{2t} \left( \|x - p\|^2 + \|z - p\|^2 - \|x - z\|^2\right) + \left\langle y - z, \nabla f(y)\right\rangle + f(z) - f(y) - \frac{L}{2} \|p - y\|^2.$
\end{description}
In addition, when $f$ is convex, we get $\left\langle y - z, \nabla f(y)\right\rangle + f(z) - f(y)\geq 0$. Therefore, inequality $(ii)$ implies that
$$ F(z) - F(p) \geq \frac{1}{2t} \left( \|x - p\|^2 + \|z - p\|^2 - \|x - z\|^2\right) - \frac{L}{2} \|p - y\|^2.$$
	\label{gradientMapping}
\end{lemma}
\begin{proof}
	We apply Lemma \ref{proximal} with $u = x - t \nabla f(y)$, $v = \prox_{tg}(u)$ and $z = w$ which leads to
	\begin{align}
		g(z) - g(p) &\geq \frac{1}{2t}\left( \|x - t \nabla f(y) - p\|^2 + \|z - p\|^2 - \|x - t \nabla f(y) - z\|^2 \right)\nonumber\\
		&=\frac{1}{2t} \left( \|x - p\|^2 + \|z - p\|^2 - \|x - z\|^2\right) + \left\langle p - z, \nabla f(y)\right\rangle.
		\label{mapping1}
	\end{align}
	Now using the descent Lemma \ref{lem1}, we have that
	\begin{align}
		f(z) - f(p) \geq -\frac{L}{2} \|p-z\|^2 - \left\langle\nabla f(z), p - z\right\rangle.
		\label{mapping2}
	\end{align}
	The first claimed inequality results from summation of (\ref{mapping1}) and (\ref{mapping2}).
	Now using the descent Lemma \ref{lem1} again, we have that
	\begin{align}
		\left\langle\nabla f(y), p - y\right\rangle \geq f(p) - f(y) - \frac{L}{2} \|p-y\|^2 .
		\label{mapping3}
	\end{align}
	Combining (\ref{mapping1}) and (\ref{mapping3}), we obtain
	\begin{align}
	g(z) - g(p) \geq &\frac{1}{2t} \left( \|x - p\|^2 + \|z - p\|^2 - \|x - z\|^2\right) + \left\langle y - z, \nabla f(y)\right\rangle + f(p) - f(y) - \frac{L}{2} \|p - y\|^2.
		\label{mapping4}
	\end{align}
	The second claimed inequality follows by adding $f(z) - f(p)$ to (\ref{mapping4}).	
	This concludes the proof.
\end{proof}
We are now ready to describe a descent property for EEG method.
\begin{proposition}[Descent condition]\label{proper1}
	For any $k\in \N$, we have
	$$F(x_k)-F(x_{k+1})\geq \frac{1}{2\alpha_k} \|x_k-x_{k+1}\|^2+\left(\frac{1}{s_k}-\frac{L}{2}-\frac{1}{2\alpha_k}\right)\|x_k-y_k\|^2+\left(\frac{1}{2\alpha_k}-\frac{L}{2}\right) \|y_k-x_{k+1}\|^2.$$
\end{proposition}
\begin{proof}
We fix an arbitrary $k \in \NN$. Applying inequality $(i)$ of Lemma~\ref{gradientMapping}, with $x=x_k$, $y=x_k$, $t = s_k$, $p=y_k$ and $z = x_k$, we obtain
\begin{equation}{\label{p1}}
F(x_k)-F(y_k)\geq \left(\frac{1}{s_k}-\frac{L}{2}\right)\|x_k-y_k\|^2.	
\end{equation}
Similarly, applying inequality $(i)$ of Lemma~\ref{gradientMapping}, with $x=x_k$, $y=y_k$, $t = \alpha_k$, $p=x_{k+1}$ and $z = y_k$, we obtain
	\begin{equation}{\label{p2}}
	F(y_k)-F(x_{k+1})\geq \frac{1}{2\alpha_k} \left(\|x_k-x_{k+1}\|^2-\|y_k-x_k\|^2 \right)+\left(\frac{1}{2\alpha_k}-\frac{L}{2}\right) \|y_k-x_{k+1}\|^2.
	\end{equation}
Combining inequalities (\ref{p1}) and (\ref{p2}), we obtain
$$F(x_k)-F(x_{k+1})\geq \frac{1}{2\alpha_k} \|x_k-x_{k+1}\|^2+\left(\frac{1}{s_k}-\frac{L}{2}-\frac{1}{2\alpha_k}\right)\|x_k-y_k\|^2+\left(\frac{1}{2\alpha_k}-\frac{L}{2}\right) \|y_k-x_{k+1}\|^2,$$
which concludes the proof
\end{proof}
\begin{remark}
	\label{rem:remarkDecrease}
	If we combine the constraint that $0<\alpha_k\leq \frac{1}{L}$ for all $k \in \NN$ with condition ({\bf C}), we deduce from Proposition \ref{proper1} that, for all $k \in \NN$,  $\frac{1}{s_k}-\frac{L}{2}-\frac{1}{2\alpha_k}\geq 0$, and
	$$F(x_k)-F(x_{k+1})\geq \frac{1}{2\alpha_k} \|x_k-x_{k+1}\|^2.$$
	Under this condition, we have that EEG is a descent method in the sense that it will produce a decreasing sequence of objective value.
\end{remark}
We now establish a second property of sequences produced by EEG method which is interpreted as a subgradient step property. We begin with a technical Lemma.
\begin{lemma}
	Assume that $(s_k, \alpha_k)_{k \in \NN}$ satisfy condition ({\bf C}). Then, for any $k\in \N$, it holds that
	\begin{align}
		\|x_{k+1}-y_{k}\|&\leq \left(\frac{1}{1-Ls_k}-\frac{s_k}{\alpha_k}\right) \|x_k-x_{k+1}\|.\label{ineq3}
	\end{align}
				\label{linkykxk}
\end{lemma}
\begin{proof}
	Denote $z_{k+1}=\prox_{s_kg}(x_k-s_k\nabla f(y_k))$, since $\prox_{s_kg}$ is 1-Lipschitz continuous, we get
	\begin{align}
	\|y_k-z_{k+1}\|&\leq \|(x_k-s_k\nabla f(y_k))-(x_k-s_k\nabla f(x_k))\|\notag\\
	&\leq Ls_k \|x_k-y_k\|,\notag
	\end{align}
where the second inequality follows from the fact that $\nabla f$ is $L$--Lipschitz continuous. Therefore,
	\begin{equation}\label{ineq1}
	\|x_k-z_{k+1}\|\geq  \|x_k-y_k\|-\|y_k-z_{k+1}\| \geq (1-Ls_k) \|x_k-y_k\|.	
	\end{equation}
Writing the optimality condition for (\ref{def2}), yields that
	\begin{equation}\label{opimalxk}
	\frac{x_k-x_{k+1}}{\alpha_k}-\nabla f(y_k)\in \partial g(x_{k+1}),
	\end{equation}	
and the convexity of $g$ implies
	$$\left\langle \frac{x_k-x_{k+1}}{\alpha_k}-\nabla f(y_k), z_{k+1}-x_{k+1}\right\rangle\leq g(z_{k+1})-g(x_{k+1}).$$
	Similarly, using the definition of $z_{k+1}$, we get
$$\left\langle  \frac{x_k-z_{k+1}}{s_k}-\nabla f(y_k), x_{k+1}-z_{k+1}\right\rangle\leq g(x_{k+1})-g(z_{k+1}).$$
Adding the last two inequalities, we obtain
	$$\left\langle \frac{x_k-z_{k+1}}{s_k}- \frac{x_k-x_{k+1}}{\alpha_k}, x_{k+1}-z_{k+1}\right\rangle\leq 0,$$
	or equivalently
	$$\left\langle \frac{x_k-z_{k+1}}{s_k}- \frac{x_k-x_{k+1}}{\alpha_k}, (x_{k+1}-x_k)+(x_k-z_{k+1})\right\rangle\leq 0.$$
	It follows that
	$$\frac{\|x_k-z_{k+1}\|^2}{s_k}+\frac{\|x_k-x_{k+1}\|^2}{\alpha_k}\leq \left(\frac{1}{s_k}+\frac{1}{\alpha_k}\right)\left\langle x_k-z_{k+1}, x_k-x_{k+1}\right\rangle.$$
Using the Cauchy-Schwarz inequality, we get
	$$\frac{\|x_k-z_{k+1}\|^2}{s_k}+\frac{\|x_k-x_{k+1}\|^2}{\alpha_k}\leq
	\left(\frac{1}{s_k}+\frac{1}{\alpha_k}\right)\|x_k-z_{k+1}\|.\|x_k-x_{k+1}\|.$$
	Since from condition ({\bf C}), $0 < s_k$, this is equivalent to
	$$\left(\|x_k-z_{k+1}\| -\|x_k-x_{k+1}\|\right) \left(\|x_k-z_{k+1}\|-\frac{s_k\|x_k-x_{k+1}\|}{\alpha_k}\right)\leq 0.$$
	This inequality asserts that the product of two terms is nonpositive. Hence one of the terms must be nonpositive and the other one must be nonnegative. From condition ({\bf C}), we have $ \frac{s_k}{\alpha_k}\leq 1$, the last term is bigger than the first one and hence must be nonnegative. This yields
	$$\frac{s_k}{\alpha_k} \|x_k-x_{k+1}\|\leq \|x_k-z_{k+1}\|\leq \|x_k-x_{k+1}\|.$$
	By combining the latter inequality with (\ref{ineq1}), we get
	\begin{equation}\label{ineq2}
	(1-Ls_k)\|x_k-y_k\|\leq\|x_k-z_{k+1}\|\leq  \|x_k-x_{k+1}\|.
	\end{equation}
	Similarly, from the definitions of $y_k, x_{k+1}$ and the convexity of $g$, we obtain that 	
	$$\left\langle  \frac{x_k-y_{k}}{s_k}-\nabla f(x_k), x_{k+1}-y_{k}\right\rangle\leq g(x_{k+1})-g(y_{k}),$$
	and
	$$\left\langle \frac{x_k-x_{k+1}}{\alpha_k}-\nabla f(y_k), y_{k}-x_{k+1}\right\rangle\leq g(y_{k})-g(x_{k+1}).$$
	Summing the last two inequalities, we have that
	$$\frac{1}{s_k}\|x_{k+1}-y_k\|^2+\left(\frac{1}{s_k}-\frac{1}{\alpha_k}\right)\langle x_{k+1}-y_{k},x_k-x_{k+1}\rangle \leq \langle x_{k+1}-y_k, \nabla f(x_k)-\nabla f(y_k)\rangle.$$
	Using the condition $0<s_k\leq \alpha_k$ and the Cauchy-Schwarz inequality, we get
	$$\frac{1}{s_k}\|x_{k+1}-y_k\|^2\leq \left(\frac{1}{s_k}-\frac{1}{\alpha_k}\right)\| x_{k+1}-y_{k}\|\|x_k-x_{k+1}\|+\|x_{k+1}-y_k\| \|\nabla f(x_k)-\nabla f(y_k)\|.$$
	Using the Lipschitz continuity of $\nabla f$, we have that
	$$ \|x_{k+1}-y_{k}\| \leq \left(1-\frac{s_k}{\alpha_k}\right)\|x_k-x_{k+1} \|+Ls_k\|x_{k}-y_k\|.$$
	Combining this inequality with (\ref{ineq2}), we obtain
	\begin{align}
	\|x_{k+1}-y_{k}\|&\leq\left(1-\frac{s_k}{\alpha_k}+\frac{Ls_k}{1-Ls_k}\right) \|x_k-x_{k+1}\|\nonumber\\
	&= \left(\frac{1}{1-Ls_k}-\frac{s_k}{\alpha_k}\right) \|x_k-x_{k+1}\|,
	\end{align}
	which is the required inequality.	
\end{proof}
We are now ready to prove the subgradient step property which is the second main element of the convergence proof.
\begin{proposition}[Subgradient step]\label{condGra}
	Assume that $(s_k, \alpha_k)_{k \in \NN}$ satisfy condition ({\bf C}). Then, for any $k\in \N$, there exists $ u_{k+1}\in \partial g(x_{k+1})$ such that
	$$\|u_{k+1}+\nabla f(x_{k+1})\|\leq \frac{L\alpha_k+(1-Ls_k)^2}{\alpha_k(1-Ls_k)}\|x_k-x_{k+1}\|.	$$
\end{proposition}
\begin{proof}				
Thanks to (\ref{opimalxk}), we deduce that there exists $u_{k+1}\in \partial g(x_{k+1})$ such that
	$$\frac{x_k-x_{k+1}}{\alpha_k}+\nabla f(x_{k+1}) -\nabla f(y_k)=u_{k+1}+\nabla f(x_{k+1}).$$
	This implies that
	$$\|u_{k+1}+\nabla f(x_{k+1})\|\leq\frac{\|x_k-x_{k+1}\|}{\alpha_k}+\|\nabla f(x_{k+1})-\nabla f(y_k)\|. $$
	Since $\nabla f$ is $L$-Lipschitz continuous, it follows that
	\begin{equation}\label{equa3}
	\|u_{k+1}+\nabla f(x_{k+1})\|\leq\frac{\|x_k-x_{k+1}\|}{\alpha_k}+L\|x_{k+1}-y_k\|.
	\end{equation}	
	Combining Lemma \ref{linkykxk} with (\ref{equa3}), we get
	\begin{equation}{\label{ineq4}}
	\|u_{k+1}+\nabla f(x_{k+1})\|\leq \frac{L\alpha_k+(1-Ls_k)^2}{\alpha_k(1-Ls_k)}\|x_k-x_{k+1}\|,
	\end{equation}
	and the result is proved.
\end{proof}
Combining Remark \ref{rem:remarkDecrease} and Proposition \ref{condGra} above, we have the following corollary which underlines the fact that EEG is actually an approximate gradient method in the sense of \cite{AttBolSva}.
\begin{corollary}\label{condC1}
	Assume that $\left(s_k,\alpha_k  \right)_{k\in\NN}$ satisfy the following
	$$({\bf C1}):  \left(s_k,\alpha_k  \right)_{k \in \NN} \text{ satisfy condition ({\bf C}) and } \alpha_k\leq \frac{1}{L}, \,\forall k\in \N.$$
	Then, for all $k \in \NN$
	\begin{description}
		\item[(i)] $F(x_{k+1})+\frac{1}{2\alpha_k}\|x_k-x_{k+1}\|^2\leq F(x_k).$
		\item [(ii)] There exists $\omega_{k+1}\in \partial F(x_{k+1})$ such that
		$$\|\omega_{k+1}\|\leq b_k\|x_k-x_{k+1}\|,$$
	\end{description}
	where,  $$0 < b_k:=\frac{L\alpha_k+(1-Ls_k)^2}{\alpha_k(1-Ls_k)} \leq b := \frac{2}{\alpha_-(1 - s^+L)}.$$
\end{corollary}

\subsection{Convergence of EEG Method under KL Assumption}\label{subsec33}
In this subsection, we analyse the convergence of EEG method in the nonconvex setting. The main result is stated in Theorem \ref{rate}, which also describes the asymptotic rate of convergence. This result is based on the assumptions that $F$ has the KL property on $\crit F$ and that  $(s_k,\alpha_k)_{k\in\NN}$ satisfy conditions ({\bf C1}) from Corollary \ref{condC1}. We will also assume that the sequence $(x_k)_{k\in\N}$ generated by EEG is bounded. This boundedness assumption is not very restrictive here, since under condition ({\bf C1}), Corollary \ref{condC1} ensures that it is satisfied for any coercive objective function. Similarly to \cite[Lemma~3.5]{BST}, we first give some properties of $F$ on the set of accumulation points of $(x_k)_{k\in \N}$.
\begin{lemma}\label{const} Assume that the sequence $(x_k)_{k\in \N}$ generated by EEG method is bounded and that $(s_k,\alpha_k)_{k\in\NN}$ satisfy condition ({\bf C1}). Let $\Omega_0$ be the set of limit points of the sequence $(x_k)_{k\in \N}$. It holds that $\Omega_0$ is compact and nonempty, $\Omega_0 \subset \crit F$, $\dist (x_k, \Omega_0)\rightarrow 0$ and  $F(\bar{x}) = \lim_{k\to\infty}F(x_k)$ for all $\bar{x} \in \Omega_0$.	
\end{lemma}
\begin{proof}
	From the boundedness assumption, it is clear that $\Omega_0$ is nonempty. In view of Corollary \ref{condC1} {\bf i)}, it follows that $(F(x_k))_{k \in \NN}$ is nonincreasing. Furthermore, $F(x_k)$ is bounded from below by $F^*$, hence there exists $\bar{F} \in \R$ such that $\bar{F} = \lim_{k\to \infty} F(x_k)$. In addition, we have
	$$\sum_{k=1}^{m}\|x_{k+1}-x_k\|^2\leq 2\alpha^+\left(F(x_1)-F(x_{m+1})\right),$$
	therefore $\sum_{k=1}^\infty\|x_{k+1}-x_k\|^2$ converges, thus $(x_{k+1}-x_k)\rightarrow 0$. We now fix an arbitrary point $x^*\in\Omega_0$, which means that there exists a subsequence $(x_{k_q})_{q \in \NN}$ of $(x_k)_{k\in \NN}$ such that $\lim_{q \to \infty} x_{k_q}=x^*$, therefore, by lower semicontinuity of $g$ and continuity of $f$,
	\begin{equation}\label{lsc}
	g(x^*)\leq \liminf_{q \to \infty} g(x_{k_q}),\, f(x^*)=\lim_{q \to \infty} f(x_{k_q}).
	\end{equation}
	From the definition of $x_{k_q}$ and condition ({\bf C1}), we get for all $q \in \NN$,
	\begin{align}
	&g(x_{k_q})+\frac{1}{2s_+}\|x_{k_q-1}-x_{k_q}\|^2+\left\langle x_{k_q}-x_{k_q-1},\nabla f(y_{k_q-1})\right\rangle\notag\\
	\leq\;&g(x_{k_q})+\frac{1}{2s_{k_q}}\|x_{k_q-1}-x_{k_q}\|^2+\left\langle x_{k_q}-x_{k_q-1},\nabla f(y_{k_q-1})\right\rangle\notag\\
	\leq\;&g(x^*)+\frac{1}{2s_{k_q}}\|x^*-x_{k_q-1}\|^2+\left\langle x^*-x_{k_q-1},\nabla f(y_{k_q -1})\right\rangle.	\notag\\	
	\leq\;&g(x^*)+\frac{1}{2s_-}\|x^*-x_{k_q-1}\|^2+\left\langle x^*-x_{k_q-1},\nabla f(y_{k_q -1})\right\rangle.	\notag	
	\end{align}
Let $q\rightarrow \infty$, it follows that $\limsup_{q \to \infty} g(x_{k_q})\leq g(x^*)$, thus, in view of (\ref{lsc}), $\lim_{q \to \infty} g(x_{k_q})=g(x^*)$, therefore $\lim_{q \to \infty} F(x_{k_q})=F(x^*)$. Since $F(x_k)$ is nonincreasing, $\lim_{q \to \infty} F(x_{k_q})=\bar{F}$, and we deduce that $F(x^*)=\bar{F}$. Since $x^*$ was arbitrary in $\Omega_0$, it holds that $F$ is constant on  $\Omega_0$.\\
\indent Now, thanks to Corollary \ref{condC1} {\bf ii)}, there exist $\omega_{k+1}\in \partial F(x_{k+1})$, such that
	$$\|\omega_{k+1}\| \leq b_k\|x_k-x_{k+1}\|.$$	
Under condition ({\bf C1}), it holds that $b_k$ remains bounded. Since $\lim_{k\to\infty}x_k-x_{k+1}= 0,$ it holds that $\omega_k\rightarrow 0$. Combining with the closedness of $\partial F$, this implies that $0\in \partial F(x^*)$, hence $x^*\in \crit F$. Since $x^*$ was taken arbitrarily in $\Omega_0$, this means that $\Omega_0\subset \crit F$. The compactness of $\Omega_0$ is implied by \cite[Lemma 5]{BST}. Combining the boundedness of $(x_k)_{k\in \N}$ and the compactness of $\Omega_0$, we deduce that $\dist(x_k,\Omega_0)\rightarrow 0$ which concludes the proof.
\end{proof}

By combining Corollary \ref{condC1}, Lemma \ref{KLunif}, \ref{const} and using the methodology of \cite[Theorem 1]{BST}, we obtain a proof of convergence of EEG method in the non-convex case.
\begin{theorem}\label{rate}
Let $(x_k)_{k\in \N}$ be a sequence generated by EEG method which is assumed to be bounded.	Suppose that $(s_k,\alpha_k)_{k\in\NN}$ satisfy condition ({\bf C1}) and that $F$ has the KL property on $\crit F$. Then, the sequence $(x_k)_{k\in \N}$ converges to $x^*\in \crit F$, moreover
	$$\sum_{i=1}^{\infty} \|x_k-x_{k+1}\|<\infty.$$
\end{theorem}
\begin{proof}
	The proof is similar to the proof of \cite[Theorem 1]{BST} and will be omitted.
\end{proof}

\begin{remark}[Convergence rate]
	When the KL desingularizing function of $F$ is of the form $\varphi(s)=cs^{1-\theta}$, where $c$ is a positive constant and $\theta \in (0,1]$, then we can estimate the rate of convergence of the sequence $(x_k)_{k\in \N}$, as follows (see \cite[Theorem 2]{attbol}).
	\begin{itemize}
		\item $\theta=0$ then the sequence $(x_k)$ converges in a~finite number of steps.
		\item  $\theta\in \left[0,\frac{1}{2}\right]$ then there exist $C>0$ and $\tau \in (0,1)$ such that
		$$\|x_k-x^*\|\leq C \tau^k, \forall k\in \N.$$
		\item  $\theta\in \left]\frac{1}{2}, 1\right[$ then there exist $C>0$ such that
		$$\|x_k-x^*\|\leq C k^{-\frac{1-\theta}{2\theta-1}}, \forall k\in \N.$$
	\end{itemize}
\end{remark}

\subsection{The Complexity of EEG in the Convex Case }\label{subsec34}
Throughout this section, we suppose that the function $f$ is convex and we focus on complexity and non asymptotic convergence rate analysis.
\subsubsection{Sublinear Convergence Rate Analysis}
We begin with a technical Lemma which introduces more restrictive step size conditions.
\begin{lemma}	
	\label{lem:stepConvex}
	Assume that $(s_k,\alpha_k)_{k\in \NN}$ satisfy the  following
	$$\text{({\bf C2}): }  (s_k,\alpha_k)_{k\in \NN} \text{ satisfy condition ({\bf C}) and } s_k\leq \frac{1}{2L}, \, \alpha_k\leq \frac{1}{L}-s_k, \, \forall k\in \N.$$
Then, for all $k \in \NN$,
	$$\frac{1}{\alpha_k}\|x_k-x_{k+1}\|^2-L\|x_{k+1}-y_k\|^2\geq 0.$$
\end{lemma}
\begin{proof}
	First, we note that if $(s_k, \alpha_k)_{k\in \NN}$ satisfy condition ({\bf C2}) then they also satisfy condition ({\bf C1}) and Proposition \ref{condGra} applies. Thanks to Lemma \ref{linkykxk}, we get
	\begin{align}
	\frac{1}{\alpha_k}\|x_k-x_{k+1}\|^2-L\|x_{k+1}-y_k\|^2&\geq \frac{1}{\alpha_k}\|x_k-x_{k+1}\|^2-L \left(\frac{1}{1-Ls_k}-\frac{s_k}{\alpha_k}\right)^2\|x_k-x_{k+1}\|^2\nonumber\\
	&=\frac{-L\alpha_k^2+(1-s_k^2L^2)\alpha_k-Ls_k^2(1-Ls_k)^2}{\alpha_k^2(1-Ls_k)^2} \|x_k-x_{k+1}\|^2.
	\label{eq:stepConvex0}
	\end{align}
	In addition, it can be checked using elementary calculation that
	$$-L\alpha_k^2+(1-s_k^2L^2)\alpha_k-Ls_k^2(1-Ls_k)^2\geq 0,$$
	is equivalent to
	\begin{align}
					(1-Ls_k)\frac{(1+Ls_k)-\sqrt{(1+Ls_k)^2-4L^2s_k^2}}{2L}\leq \alpha_k\leq (1-Ls_k)\frac{(1+Ls_k)+\sqrt{(1+Ls_k)^2-4L^2s_k^2}}{2L}.
					\label{eq:stepConvex1}
	\end{align}
	Note that, for $0\leq b\leq a$ then $a-b\leq \sqrt{a^2-b^2}$. Using this inequality, with the condition $2Ls_k\leq 1$, we get $\displaystyle (1+Ls_k)-2Ls_k\leq \sqrt{(1+Ls_k)^2-4L^2s_k^2}$. Thus,
	\begin{align*}
	\displaystyle (1-Ls_k)\frac{(1+Ls_k)-\sqrt{(1+Ls_k)^2-4L^2s_k^2}}{2L}&\leq (1-Ls_k)\frac{\left[(1+Ls_k)-(1-Ls_k)\right]}{2L}\\
	&=(1-Ls_k)s_k\leq s_k,
	\end{align*}	
	and
	$$(1-Ls_k)\frac{(1+Ls_k)+\sqrt{(1+Ls_k)^2-4L^2s_k^2}}{2L}\geq (1-Ls_k)\frac{(1+Ls_k)+(1-Ls_k)}{2L}=\frac{1}{L}-s_k.$$
	Condition ({\bf C2}) ensures that $s_k \leq \alpha_k \leq \frac{1}{L} - s_k$ and hence identity (\ref{eq:stepConvex1}) holds and (\ref{eq:stepConvex0}) implies that
	$$\frac{1}{\alpha_k}\|x_k-x_{k+1}\|^2-L\|x_{k+1}-y_k\|^2\geq 0, \, \forall k\in \N.$$
\end{proof}
With a similar method as in \cite{BT08}, we prove a sublinear convergence rate for $\left(F(x_k)\right)_{k\in \NN}$ in the convex case.
\begin{theorem}[Complexity of EEG method] Let $(x_k)_{k\in \N}$ be a sequence generated by EEG method. Suppose that $(s_k, \alpha_k)_{k\in \NN}$ satisfy condition ({\bf C2}) and that $f$ is convex. Then, for any $x^*\in \argmin F$, we have
	$$F(x_{m})-F(x^*)\leq \frac{1}{2m \alpha_-}\|x_0-x^*\|^2,\, \forall m\in \N^*.$$
\end{theorem}
\begin{proof}			
	We first fix arbitrary $k \in \NN$ and $x^*\in\argmin F$. Since $f$ is convex, applying inequality $(ii)$ of Lemma \ref{gradientMapping} with $x = x_k$, $y = y_k$, $t = \alpha_k$, $p = x_{k+1}$ and $z = x^*$, we obtain
$$F(x^*)-F(x_{k+1})\geq \frac{1}{2\alpha_k}\left(\|x^*-x_{k+1}\|^2-\|x^*-x_k\|^2 \right)+\frac{1}{2\alpha_k}\|x_k-x_{k+1}\|^2-\frac{L}{2}\|x_{k+1}-y_k\|^2.$$
Using the fact that $F(x_k)$ is noninreasing and bounded from bellow by $F(x^*)$, it follows from Lemma \ref{lem:stepConvex} that
	$$0\geq F(x^*)-F(x_{k+1})\geq \frac{1}{2\alpha_k}\left(\|x^*-x_{k+1}\|^2-\|x^*-x_k\|^2 \right)\geq \frac{1}{2\alpha_-}\left(\|x^*-x_{k+1}\|^2-\|x^*-x_k\|^2 \right).$$
Summing this inequality for $k=0,\cdots, m-1$ gives
\begin{align}
	\label{complex1}
	mF(x^*)-\sum_{k=1}^{m}F(x_{k}) &\geq \frac{1}{2\alpha_{-}} (\|x^*-x_{m}\|^2-\|x^*-x_{0}\|^2).
\end{align}

Coming back to Corollary~\ref{condC1}, it is easy to see that the sequence $(F(x_k))_{k\in \N}$ is nonincreasing, then $\displaystyle \sum_{k=1}^{m}F(x_{k})\geq mF(x_m)$.
	Combining with (\ref{complex1}), we get
		$$m\left(F(x^*)-F(x_{m})\right)\geq \frac{1}{2\alpha_-} \left(\|x^*-x_{m}\|^2-\|x^*-x_0\|^2\right).$$
		It follows that
	$$F(x_{m})-F(x^*)\leq \frac{1}{2m\alpha_-} \|x^*-x_0\|^2, \, \forall m\in \N^*.$$
\end{proof}
\subsubsection{Small-Prox Type Result under KL Property}
We now study the complexity of EEG method when $F$ has, in addition, the KL property on $\crit F$. First, using the convexity of $f$, Proposition~\ref{proper1}, can be improved by using the following result.

\begin{proposition}\label{descentCD2}
Assume that $f$ is convex and $(s_k, \alpha_k)_{k\in \NN}$ satisfy condition ({\bf C}), then for all $k\in \N$, we have
	$$F(x_k)-F(x_{k+1})\geq c_k\|x_k-x_{k+1}\|^2,$$
	where
	$$c_k:= \frac{1}{\alpha_k}-\frac{L}{2}\left(\frac{1}{1-Ls_k}- \frac{s_k}{\alpha_k}\right)^2.$$
\end{proposition}
\begin{proof}
Fix an arbitrary $k \in \NN$. Since $f$ is convex, applying inequality $(ii)$ of Lemma~\ref{gradientMapping}, with $x = x_k$, $y = y_k$, $t = \alpha_k$, $p=x_{k+1}$ and $z = x_k$, we get
	\begin{equation}\label{descentcd3}
	F(x_k)-F(x_{k+1})\geq  \frac{1}{\alpha_k}\|x_k-x_{k+1}\|^2-\frac{L}{2}\|x_{k+1}-y_k\|^2.
	\end{equation}
	Combining inequality (\ref{descentcd3}) with Lemma \ref{linkykxk}, we get the desired result,
	$$F(x_k)-F(x_{k+1})\geq\left[ \frac{1}{\alpha_k}-\frac{L}{2}\left(\frac{1}{1-Ls_k}-\frac{s_k}{\alpha_k}\right)^2\right] \|x_k-x_{k+1}\|^2.$$
\end{proof}
We now consider another step size condition.
\begin{lemma}
	\label{lem:stepSize3}
	Suppose that $s_k, \, \alpha_k$ satisfy the following condition
	$$({\bf C3})\begin{cases}
	s_k, \alpha_k \text{ satisfy condition ({\bf C})}\notag\\
	s_k\leq \frac{\sqrt{5}-1}{2L},\, \text{ and } \alpha_k\leq \frac{2}{L}-2s_k-(1-Ls_k)Ls_k^2, \, \forall k\in \N. \notag
	\end{cases}$$
	Then, for all $k \in \NN$,
	$$ \frac{1}{\alpha_k}-\frac{L}{2}\left(\frac{1}{1-Ls_k}- \frac{s_k}{\alpha_k}\right)^2\geq C:=\frac{L^3s_-^2(1+Ls_-)}{2(2-L^2s_-^2)^2(1-Ls_-)}>0.$$	
\end{lemma}
\begin{remark}
	Before starting the proof, we make a comment on the restriction $s_k\leq \frac{\sqrt{5}-1}{2L}$, which is only presented here to ensure consistency of condition ({\bf C}) and condition ({\bf C3}). By analysing a degree three polynomial, one can check that
	$$Ls_k\leq 2-2Ls_k-(1-Ls_k)L^2s_k^2$$
	if and only if
	$$Ls_k\in \left[-\frac{\sqrt{5}+1}{2},\frac{\sqrt{5}-1}{2}\right] \cup \left[2, +\infty\right[.$$
	Hence the bound $s_k \leq \frac{\sqrt{5}-1}{2L}$ is a necessary condition to ensures that $s_k \leq \frac{2}{L}-2s_k-(1-Ls_k)Ls_k^2$. This upper limit on $s_k$ could be removed from condition ({\bf C3}), but then it would be enforced implicitly by the combination of conditions ({\bf C}) and ({\bf C3}) which results in $s_k \leq \alpha_k \leq \frac{2}{L}-2s_k-(1-Ls_k)Ls_k^2$. We preferred to write it explicitly.
\end{remark}
\begin{proof}
	We fix an arbitrary $k \in \NN$. Set
\begin{align}\label{eq:stepSize30}
	\alpha_k^+&=\frac{2}{L}-2s_k-(1-Ls_k)Ls_k^2 \\
	Q(u)&=u-\frac{1}{2}\left(\frac{1}{1-Ls_k}-Ls_ku\right)^2,
\end{align}
	where one can think of $u$ satisfying $u=\frac{1}{L\alpha_k}\in \left[\frac{1}{L\alpha_k^+}, \frac{1}{Ls_k}\right]$.
	The maximum of $Q(u)$ is attained for $u=\frac{1}{(1-Ls_k)L^2s_k^2} \geq \frac{1}{Ls_k}$, and the inequality stands because $Ls_k \leq 1$ and $1 - Ls_k \leq 1$. Note that conditions ({\bf C}) and ({\bf C3}) ensure that $\alpha_k^+ \geq s_k$, hence $Q$ is increasing on $\displaystyle\left[\frac{1}{L\alpha_k^+}, \frac{1}{Ls_k}\right]$. Combining conditions ({\bf C}) and ({\bf C3}), we have that $s_k \leq \alpha_k \leq \alpha_k^+$ and therefore,
	\begin{equation}\label{eq:stepSize31}
	LQ\left(\frac{1}{L\alpha_k}\right) = \frac{1}{\alpha_k}-\frac{L}{2}\left(\frac{1}{1-Ls_k}- \frac{s_k}{\alpha_k}\right)^2\geq LQ\left(\frac{1}{L\alpha_k^+}\right).
	\end{equation}
	We now turn to algebraic manipulations to compute $LQ\left(\frac{1}{L\alpha_k^+}\right)$. First we expand and reduce to common denominator.
	\begin{align}
	\label{eq:stepSize32}
					&LQ\left(\frac{1}{L\alpha_k^+}\right)\\
					=\;&L\left(\frac{1}{L\alpha_k^+}-\frac{1}{2}\left(\frac{1}{1-Ls_k}-Ls_k \frac{1}{L\alpha_k^+}\right)^2  \right)\nonumber\\
					=\;&\frac{L}{2(1 - Ls_k)^2 (L\alpha_k^+)^2} \left( 2L\alpha_k^+ (1 - Ls_k)^2 - (L\alpha_k^+)^2 + 2 Ls_k (1- Ls_k) L\alpha_k^+ - L^2s_k^2(1-Ls_k)^2\right)\nonumber\\
					=\;&\frac{L}{2(1 - Ls_k)^2 (L\alpha_k^+)^2} \left( - (L\alpha_k^+)^2 + 2 (1- Ls_k) L\alpha_k^+ - L^2s_k^2(1-Ls_k)^2\right).\nonumber
	\end{align}
	We now use the expression of $\alpha_k^+$ given in (\ref{eq:stepSize30}) and expand the expression in (\ref{eq:stepSize32}) by using $L \alpha_k^+ = (1 - Ls_k)(2 - L^2 s_k^2)$.
	\begin{align}
	\label{eq:stepSize33}
					&LQ\left(\frac{1}{L\alpha_k^+}\right) \nonumber\\
					=\;&\frac{L}{2(1 - Ls_k)^4 (2 - L^2 s_k^2)^2} \left( - (1 - Ls_k)^2(2 - L^2 s_k^2)^2 + 2 (1- Ls_k)^2(2 - L^2 s_k^2) - L^2s_k^2(1-Ls_k)^2\right)\nonumber\\
					=\;&\frac{L}{2(1 - Ls_k)^2 (2 - L^2 s_k^2)^2} \left( - (2 - L^2 s_k^2)^2 + 2 (2 - L^2 s_k^2) - L^2s_k^2\right)\nonumber\\
					=\;&\frac{L}{2(1 - Ls_k)^2 (2 - L^2 s_k^2)^2} \left( - 4 + 4 L^2 s_k^2 -  L^4 s_k^4+ 4 - 2 L^2 s_k^2 - L^2s_k^2\right)\nonumber\\
					=\;&\frac{L}{2(1 - Ls_k)^2 (2 - L^2 s_k^2)^2} \left(  L^2 s_k^2 -  L^4 s_k^4\right)\nonumber\\
					=\;&\frac{L^3s_k^2}{2(1 - Ls_k)^2 (2 - L^2 s_k^2)^2} \left( 1 -  L^2 s^2_k\right)\nonumber\\
					=\;&\frac{L^3s_k^2}{2(1 - Ls_k) (2 - L^2 s_k^2)^2} \left( 1+L s_k\right).
	\end{align}
	Combining (\ref{eq:stepSize31}) and (\ref{eq:stepSize33}), we obtain
	\begin{align*}
		\frac{1}{\alpha_k}-\frac{L}{2}\left(\frac{1}{1-Ls_k}- \frac{s_k}{\alpha_k}\right)^2\nonumber &\geq LQ\left(\frac{1}{L\alpha_k^+}\right)\\
		&=\frac{L^3s_k^2(1+Ls_k)}{2(2-L^2s_k^2)^2(1-Ls_k)}\\
		&\geq \frac{L^3s_-^2(1+Ls_-)}{2(2-L^2s_-^2)^2(1-Ls_-)}=C,
	\end{align*}
	which is the desired result.
\end{proof}
	
We can check that, when condition ({\bf C3}) is satisfied, one has
$$0<b_k=\frac{L\alpha_k +(1-Ls_k)^2}{\alpha_k(1-Ls_k)} \leq \frac{2-2Ls_k+(1-Ls_k)^2}{\alpha_k(1-Ls_k)}=\frac{3-Ls_k}{\alpha_k}\leq B= \frac{3}{\alpha_{-}}.$$
Combining this with Proposition~\ref{condGra},   \ref{descentCD2} and Lemma \ref{lem:stepSize3}, we obtain the following corollary.
\begin{corollary}\label{cor:gradMethod2}
	Suppose that $(s_k,\alpha_k)_{k\in \NN}$ satisfy condition  ({\bf C3}) and that $f$ is convex, then
	\begin{description}
		\item[(i)] $F(x_{k+1})+C\|x_k-x_{k+1}\|^2\leq F(x_k), \, \forall k\in \N.$
		\item [(ii)] There exists $\omega_{k+1}\in \partial F(x_{k+1})$ such that
		$$\|\omega_{k+1}\|\leq B\|x_k-x_{k+1}\|,\,\forall k\in \N.$$
	\end{description}
	where $C$ is given in Lemma \ref{lem:stepSize3} and $B = \frac{3}{\alpha_-}$.
\end{corollary}
We now consider the complexity for EEG method under the nonsmooth KL inequality in the form of a \textit{small prox} result as in \cite{eb&kl}. First, we recall some definitions from \cite{eb&kl}. Let $0<r_0:=F(x_0)<\bar r$, we assume that $F$ has the KL property on $[0<F<\bar r]$ with desingularizing function $\f\in \cali K(\bar r)$.

Set $\beta_0:=\f(r_0)$ and consider the function $\psi:=(\varphi\vert_{[0,r_0]})^{-1}:[0,\beta_0]\to [0,r_0]$, which is increasing and convex. We add the assumption that $\psi'$ is Lipschitz continuous (on $[0,\beta_0]$) with constant $\ell>0$ and $\psi'(0)=0$.

Set
$$\zeta:=\frac{\sqrt{1+2\ell \,C \,B^{-2}}-1}{\ell}.$$
Starting from $\beta_0$, we define the sequence $(\beta_k)_{k\in \N}$ by
\begin{align}
\beta_{k+1}:&=\argmin\left\{\psi(u)+\frac{1}{2\zeta} (u-\beta_k)^2:u\geq 0\right\}\notag\\
&=\prox_{\zeta\psi} (\beta_k).\notag
\end{align}
It is easy to prove that $\beta_k$ is decreasing and converges to zero. By continuity, $\lim\limits_{k\to\infty}\psi(\beta_k)=0$.

Now, applying the result of \cite[Theorem 17]{eb&kl}, we have the complexity of EEG method in the form of a \textit{small prox} result.
\begin{theorem}[Complexity of EEG method]
Let $(x_k)_{k\in \N}$ be a sequence generated by EEG method. Assume that $f$ is convex and $(s_k, \alpha_k)_{k\in \NN}$ satisfy condition ({\bf C3}). Then, the sequence $(x_k)_{k\in \N}$ converges to $x^*\in\argmin F$, and
	$$\sum_{i=1}^{\infty} \|x_k-x_{k+1}\|<\infty,$$
	moreover,
	\begin{align}
	& F(x_k)-F^*  \leq  \psi(\beta_k),  \quad\forall k\geq 0,\notag\\
	&\|x_k-x^*\|  \leq  \frac{B}{C} \beta_k+\sqrt{\frac{\psi(\beta_{k-1})}{C}}, \quad\forall k\geq 1,\notag 
	\end{align}
	where $B$ and $C$ are given in Corollary \ref{cor:gradMethod2}.	
\end{theorem}

\section{Numerical Experiment }\label{sec4}
In this section, we compare the EEG method with standard algorithms in numerical optimization: Forward-Backward and FISTA. We describe the problem of interest, details about exact line search in this context and numerical results.
\subsection{$\ell_1$ Regularized Least Squares}
We let $A \in \RR^{p \times n}$ be a real matrix, $b \in \RR^n$ be a real vector and $\lambda > 0$ be a scalar, all of them given and fixed. Following the notations of the previous section, we define $f \colon x \mapsto \frac{1}{2} \|Ax - b\|_2^2$ and $g \colon x \mapsto \lambda \|x\|_1$ (the sum of absolute values of the entries). With these notations, the optimization problem ({\bf P}) becomes
\begin{align}
\label{eq:L1LeastSquares}
\min_{x \in \RR^n}\;& \frac{1}{2}\|Ax - b\|_2^2 + \lambda \|x\|_1.
\end{align}
Solutions of problem of the form of (\ref{eq:L1LeastSquares}) (as well as many extensions) are extensively used in statistics and signal processing \cite{tibshirani1996regression,chen1999atomic}. For this problem, we introduce the proximal gradient mapping, a specialization of the proximal gradient step to problem (\ref{eq:L1LeastSquares}). This is the main building block of all the algorithms presented in the numerical experiment.
\begin{align}
\label{eq:proxMapping}
\begin{array}{clcl}
p \colon& \RR^n \times \RR_+ &\mapsto& \RR^n\\
&(x, s)&\mapsto& S_{s\lambda}(x - s \nabla f(x)).
\end{array}
\end{align}
where $S_a$ ($a \in \RR_+$) is the soft-thresholding operator which acts coordinatewise and satisfies for $i = 1, 2\ldots,n$
\begin{align*}
[S_{a}(x)]_i =
\begin{cases}
0, & \text{if } |x_i| \leq a.\\
x_i - a {\rm sign}(x_i), &\text{otherwise}.
\end{cases}
\end{align*}
\subsection{Exact Line Search}
One intuition behind Extragradient-Method for optimization is the use of an additional iteration as a guide or a scout to provide an estimate of the gradient that better suits the geometry of the problem. This should eventually translate to taking larger steps leading to faster convergence. In order to evaluate Extragradient-Method, we need a mechanism which would allow us to take larger steps when this is beneficial. One such mechanism is exact line search. This strategy is not widely used because of its computational overhead. In this section, we briefly describe a strategy which allows to perform exact line search efficiently in the context of $\ell_1$-regularized least squares. As far as we know, this approach has not been described in the literature. Furthermore, this strategy may be extended to more general least squares problems with nonsmooth regularizers. For the rest of this section, we assume that $x \in \RR^n$ is fixed. We heavily rely on the two simple facts:
\begin{itemize}
	\item The mapping $s \to p(x, s)$ is continuous and piecewise affine.
	\item The objective function $x \mapsto f(x) + g(x)$ is continuous and piecewise quadratic.
\end{itemize}
We consider the following function
\begin{align*}
q_x \colon \RR_+ &\to \RR\\
\alpha &\to f(p(x, \alpha)) + g(p(x, \alpha)).
\end{align*}
It can be deduced from the properties of $f$, $g$ and $p$ that $q_x$ is continuous and piecewise quadratic. In classical implementation of proximal splitting methods, the step-size parameter $\alpha$ is a well chosen constant which depends on the problem, or alternatively it is estimated using backtracking. The alternative which we propose is to choose the step-size parameter $\alpha$ minimizing $q_x$. Since $q_x$ is a one dimensional piecewise quadratic function, then we only need to know its expression between the values of $\alpha$ which constitute breakpoints where the quadratic expression of the function $q_x$ changes, \textit{i.e.} points where $q_x$ is not differentiable.

The nonsmooth points of $q_x$ are given by the following set
\begin{align*}
\mathcal{D}_x = \left\{\frac{x_i}{\frac{\partial f(x)}{\partial x_i} - \lambda}, \frac{x_i}{\frac{\partial f(x)}{\partial x_i} + \lambda} \right\}_{i=1}^n \cap \RR_+
\end{align*}
and correspond to limiting values for which coordinates of $p(x, \alpha)$ are null. We assume that the elements of $\mathcal{D}_x$ are ordered nondecrasingly (letting potential ties appear several times). The comments that we have made so far lead to the following.
\begin{itemize}
	\item $\mathcal{D}_x$ contains no more than $2n$ elements.
	\item Given $x$ and $\lambda$, computing $\mathcal{D}_x$ is as costly as computing $\nabla f$.
	\item $q_x$ is quadratic between two consecutive elements of $\mathcal{D}_x$.
\end{itemize}
In order to minimize $q_x$, the only task that should be performed is to keep track of its value (or equivalently of its quadratic expression) between consecutive elements of $\mathcal{D}_x$. Here, we can use the fact that elements of $\mathcal{D}_x$ corresponds to values of $\alpha$ for which one coordinate of $p(x, \alpha)$ goes to zero or becomes active (non-zero). A careful implementation of the minimization of $q_x$ amounts to sort the values in $\mathcal{D}_x$, placing them in increasing order, keeping track of the corresponding quadratic expression and the minimal value. We provide a few details for completeness.
\begin{itemize}
	\item The vector $d_x(s) := \left( \frac{\partial [p(x,s)]_i}{\partial s} \right)_{i=1}^n \in \RR^n$ is constant between consecutive elements of $\mathcal{D}_x$. Furthermore the elements of $\mathcal{D}_x$ (counted with multiple ties) corresponds to value of $\alpha$ for which a single coordinate of $d_x(s)$ is modified.
	\item Suppose that $\alpha_1 < \alpha_2$ are two consecutive elements of $\mathcal{D}_x$. Then for all $\alpha \in [\alpha_1, \alpha_2]$, letting $d_x(\alpha) = d$ on this segment, we have $p(x, \alpha) = p(x, \alpha_1) +(\alpha - \alpha_1)d$, hence,
	\begin{align*}
	&\frac{1}{2}\|Ap(x, \alpha) - b\|_2^2 + \lambda \|p(x, \alpha)\|_1 \\
	=\;&\frac{1}{2}\|Ap(x, \alpha_1) - b\|_2^2 + \lambda \|p(x, \alpha_1)\|_1\\
	&+ \frac{(\alpha - \alpha_1)}{n} \left\langle Ad, Ax - b \right\rangle+ \frac{(\alpha-\alpha_1)^2}{2n} \|Ad\|_2^2 + \lambda (\alpha- \alpha_1) \left\langle \bar{d}, d\right\rangle,\\
	\end{align*}
where $\bar{d} \in \RR^p$ is a vector which depends on the sign pattern of $p(x, \alpha_1)$ and $d$.
\item For $\alpha = \alpha_2$, the sign pattern of $p(x, \alpha_2)$ and the corresponding value of $d$  and $\bar{d}$ (for the next interval) are modified only at a single coordinate, the same for the three of them. In other words, updating the quadratic expression of $q_x$ at $\alpha_2$ only requires the knowledge of this coordinate, the value of the corresponding column in $A$ and can be done by computing inner products in $\RR^p$. This requires $O(p)$ operations.
	\item Given these properties, we can perform minimization of $q_x$ by an active set strategy, keeping track only of the sign pattern of $p(x,\alpha)$, the value of $\left\langle\bar{d},d\right\rangle$, the value of $Ad$, $Ap(x, \alpha) - b$ and $\|p(x, \alpha)\|_1$  which cost is of the order of $O(p)$. This should not be repeated more than $2n$ times.
\end{itemize}
Using this active set procedure provides the quadratic expression of $q_x$ for all intervals represented by consecutive values in $\mathcal{D}_x$. From these expressions, it is not difficult to compute the global minimum of $q_x$. The overall cost of this operation is of the order of $O(np)$ plus the cost of sorting $2n$ elements in $\RR$. This is comparable to the cost of computing the gradient of $f$. Hence in this specific setting, performing exact line search does not add much overhead in term of computational cost compared to existing step-size strategies.

\begin{figure}[t]
	\centering
	\includegraphics[width=\textwidth]{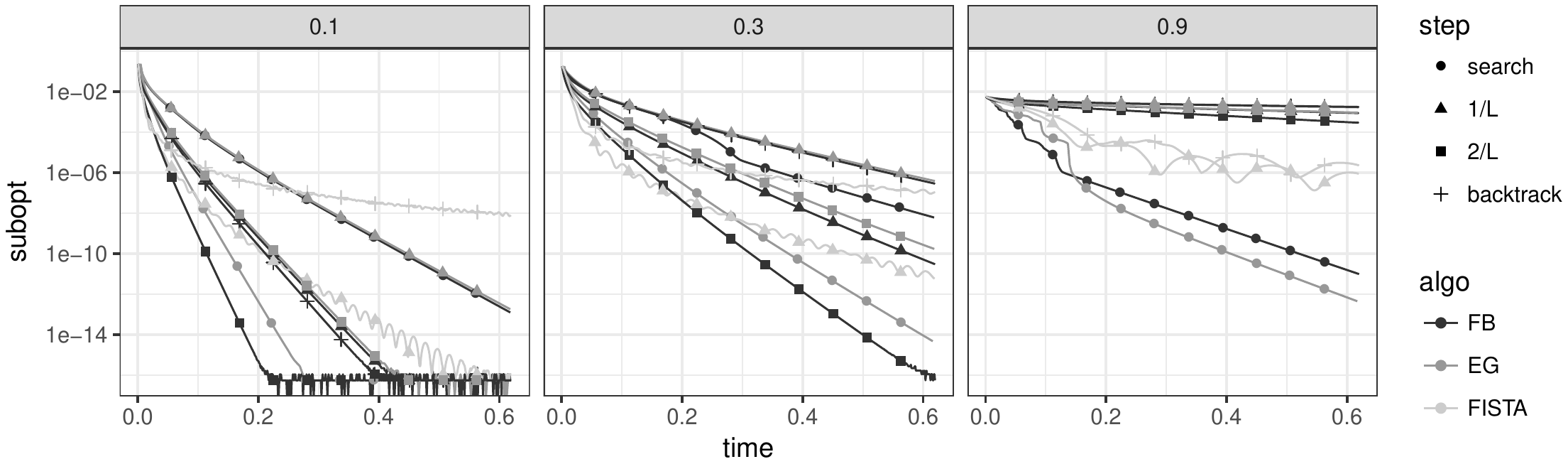}
	\caption{Suboptimality ($F(x_k) - F^*$) as a function of time for simulated $\ell_1$ regularized least squares data. FB stands for Forward-Backward and EG for Extra-Gradient. The color is related to the algorithm used and the dots are related to the step size used. Different windows show different values of the parameter $\delta$ (see the main text for a precise description). On the left we have a well conditioned problem and on the right the conditioning is much worse. For $\delta = 0.1,\,0.3,\,0.9$, the condition number of the matrix $A$ are approximately $6$, $15$ and $300$ respectively.}
	\label{fig:computTime}
\end{figure}

\subsection{Simulation and Results}
We generate a matrix $A$ and vector $b$ using the following process.
\begin{itemize}
	\item Set $n = 600$ and $p = 300$.
	\item Set $A = D X$ where $X$ has standard Gaussian independent entries and $D$ is a diagonal matrix which $i$-th diagonal entry is $\frac{1}{i^\delta}$ where $\delta$ is a positive parameter controlling the good conditioning of the matrix $A$ (the smaller $\delta$, the better).
	\item Choose $b$ with independant Gaussian entries.
	\item We set $\lambda = 1/n\simeq 0.001$.
\end{itemize}
We compare the forward-backward splitting algorithm, FISTA \cite{BT08} and the proposed extragradient method with different step size rules ($L$ is the Lipschitz constant of $f$ computed from the singular values of $A$).
\begin{itemize}
	\item A step of size $1/L$.
	\item A step of size $2/L$.
	\item A step given by backtracking line search (see e.g. \cite{BT08}). The original guess for $L$ is chosen to be $1$ and the multiplicative parameter is $1.2$.
	\item A step given by exact line search as described in the previous section.
\end{itemize}
For the extragradient method, we always choose $s=1/L$ and determine $\alpha$ by the chosen step-size rule. For FISTA algorithm, we do not implement the $2/L$ and exact line search step size rules as they produce diverging sequences.
The exact line search active set procedure is implemented in compiled \texttt{C} code in order keep a reasonable level of efficiency compared to linear algebra operations which have efficient implementations. The algorithms are initialized at the origin. We keep track of decrease of the objective value, the iteration counter $k$ and the total time spent since initialization. The iteration counter is related to analytical complexity while the total time spent is related to the arithmetical complexity (see the introduction in \cite{nes} for more details). Comparing algorithms in term of analytical complexity does not reflect the fact that iterations are more costly for some of them compared to others so we only focus on arithmetical complexity which in our case is roughly proportional to computational time.

Computational times for a generic LASSO problem are presented in Figure \ref{fig:computTime} for $\delta = 0.1,\,0.3,\,0.9$. The main comments are as follows:
\begin{itemize}
	\item For well conditioned problems, the forward-backward algorithm with step size $2/L$ performs the best. This is not the case for the less well conditioned problem where exact line search method shows some advantage.
	\item The extragradient method with exact line search performs reasonably well, independently of the conditioning of the problem.
	\item FISTA algorithm is outperformed by other methods in terms of asymptotic convergence. Furthermore,  FISTA's performance is very sensitive to step-size tuning.
\end{itemize}

This experiment illustrates that exact line search can improve performances for ill-conditioned problems and that the proposed extragradient method is able to take advantage of it, independently of the problem's conditioning. This observation is based on a ``generic'' instance of the LASSO problem. Further experiments on real data are required to confirm generality of the observation. This is a matter of future research.

\section{Conclusions}
In this paper, we presented an extension of extragradient method, EEG, and used it to tackle the problem of minimizing the sum of two functions. Under step size conditions, we showed that EEG is a first order descent method. By using the KL inequality, we obtained the convergence of the sequence produced by EEG method and estimated the complexity of EEG method via the small-prox method.  In the convex setting, we obtained a classical sublinear convergence rate for the objective function value.  Finally, we described an exact line search strategy for the $\ell_1$ regularized least square problem and conducted numerical comparisons with existing algorithms on a generic instance of the LASSO problem.

\paragraph{Acknowledgement of Support and Disclaimer}
This work is sponsored by a grant from the Air Force Office of Scientific Research, Air Force Material Command (grant number FA9550-15-1-0500). Any opinions, findings and conclusions or recommendations expressed in this material are those of the authors and do not necessarily reflect the views of the United States Air Force Research Laboratory. The collaboration with Emile Richard mostly took place during his postdoctoral stay in Mines ParisTech, Paris, France in 2013.

The authors would like to thank Professor J\'er\^ome Bolte for his suggestions, the associate editor and anonymous referee for helpful remarks which helped improve the quality of this manuscript. 
\bibliographystyle{unsrt}
\bibliography{arxiv_11}
\end{document}